\documentclass[11pt,a4paper]{article}
\usepackage{amscd}
\usepackage{enumerate}

\usepackage{amsmath, amssymb, latexsym}
\usepackage{bbold}
\usepackage{amsfonts}
\usepackage{amsthm}
\usepackage{relsize}
\usepackage{setspace}
\usepackage{geometry}
\usepackage{url}
\usepackage{xspace}
\usepackage{tocloft}
\usepackage{graphics}
\usepackage{graphicx}
\usepackage{lscape}
\usepackage{microtype}
\usepackage{ulem}

\usepackage[usenames, dvipsnames]{color}
\usepackage[utf8]{inputenc}
\usepackage{tikz}

\usepackage[pagebackref=true]{hyperref}
\usepackage[alphabetic]{amsrefs}
\usepackage[english]{babel}

\usepackage{authblk}

\newtheorem{theorem}{Theorem}[section]
\newtheorem{proposition}[theorem]{Proposition}

\newtheorem{corollary}[theorem]{Corollary}
\newtheorem{lemma}[theorem]{Lemma}

\theoremstyle{definition}
\newtheorem{definition}[theorem]{Definition}

\newtheorem{remark}[theorem]{Remark}

\newtheorem{convention}[theorem]{Convention}

\theoremstyle{problem}

\newtheorem{question}[theorem]{Question}

\newcommand{\Aut}{\mathrm{Aut}}

\newcommand{\RR}{\mathbb{R}}

\newcommand{\QQ}{\mathbb{Q}}
\newcommand{\CC}{\mathbb{C}}

\newcommand{\Min}{\mathrm{Min}}

\newcommand{\bd}{\partial}

\newcommand{\id}{\operatorname{id}}

\newcommand{\SL}{\operatorname{SL}}

\newcommand{\T}{T_{d}}
\newcommand{\G}{\operatorname{\mathbb{G}}}
\newcommand{\dist}{\operatorname{dist}}

\newcommand{\Sym}{\operatorname{Sym}}
\newcommand{\supp}{\operatorname{supp}}

\def\og{\leavevmode\raise.3ex\hbox{$\scriptscriptstyle\langle\!\langle$~}}
\def\fg{\leavevmode\raise.3ex\hbox{~$\!\scriptscriptstyle\,\rangle\!\rangle$}}

\title{The universal group of Burger--Mozes and the Howe--Moore property}

\author{Corina Ciobotaru\thanks{Partially supported  by the FRIA; corina.ciobotaru@gmail.com}}

\date{January 3, 2021}

\begin{document}

\renewcommand{\thefootnote}{}

\newcounter{qcounter}

\maketitle

\begin{abstract}
By constructing  a new unitary representation we prove the universal group $U(F)^+$ of Burger--Mozes does not have the Howe--Moore property when  $F$ is primitive but not $2$-transitive. It is well known $U(F)^+$ does have this property when $F$ is $2$-transitive. Along the way, we give a characterization of the universal group, when $F$ is primitive, to have the Howe--Moore property, and also prove $U(F)^+$ has the relative Howe--Moore property. These two results are a consequence of a strengthening of Mautner's phenomenon for locally compact groups acting on d-regular trees and having Tits' independence property. 
\end{abstract}

\section{Introduction}

The motivation of this paper is to answer the following question proposed several years ago by Marc Burger and Shahar Mozes, and Pierre-Emmanuel Caprace. This question is not explicitly written somewhere, but it was the main 
topic of the author's PhD thesis:
\begin{question}
\label{que::main_question}
Let $G$ be a locally compact, topologically simple group acting continuously and properly by type-preserving automorphisms on a $d$-regular tree $T_{d}$. If $G$ has the Howe--Moore property is it true that G must act $2$--transitively on the boundary of the tree $T_{d}$?
\end{question}

Recall the following. Let $G$ be a locally compact group and $(\pi, \mathcal{H})$ a (strongly continuous) unitary representation  $\pi :G \to \mathcal{U}(\mathcal{H}) $ on a (infinite dimensional) complex Hilbert space $(\mathcal{H}, \left\langle \cdot , \cdot \right\rangle)$. For $v,w \in \mathcal{H}$ the associated \textbf{$(v,w)$-matrix coefficient} is the map $c_{v,w} : G \to \mathbb{C}$ given by $c_{v,w}(g):=\left\langle \pi(g)v ,w \right\rangle$. We say $c_{v,w}$ \textbf{vanishes at infinity} if, for every $\epsilon >0$, the subset $\{g \in G \; \vert \; \vert c_{v,w}(g)\vert \geq \epsilon \}$ is compact in $G$; equivalently,  $\lim\limits_{g \to \infty} c_{v,w}(g)  =0$, where $\infty$ represents the one-point compactification of the locally compact group $G$.

\begin{definition}
\label{def::HM}
Let $G$ be a locally compact group. We say $G$ has \textbf{the Howe--Moore property} if for any unitary representation of $G$ that is without non-zero $G$--invariant vectors, all its matrix coefficients vanish at infinity. 
\end{definition}

Just to recall, the Howe--Moore property \cite{HM79} is a harmonic analytic property that has far-reaching applications to many areas of mathematics, such as homogeneous dynamics, geometry, or number theory. To list some, it has notable consequences regarding discrete subgroups of semi-simple Lie groups, Mostow's rigidity theorem, mixing and equidistribution properties, Ratner's theorems on unipotent flows, or dynamical systems endowed with an invariant measure. 

All known examples of locally compact groups having the Howe--Moore property are: 
\begin{enumerate}
\item
 connected, non-compact, simple real Lie groups, with finite center \cite{HM79}, e.g., $\SL(n, \RR)$
 \item
  isotropic simple algebraic groups over non Archimedean local fields \cite{HM79}, e.g., $\SL(n,\QQ_p)$
 \item
 closed, topologically simple subgroups of $\Aut(T)$ with a $2$-transitive action on the boundary of the bi-regular tree $T$, that has valence $\geq 3$ at every vertex, \cite{BM00a}, e.g., the universal group $U(F)^+$ of Burger--Mozes, when F is $2$-transitive.  
 \end{enumerate}
 An important fact the above mentioned groups have in common is their polar decomposition $KA^+K$, where $K$ is a maximal compact subgroup and $A^+$ is an abelian maximal sub-semi-group, the latter being sufficient to imply the Howe--Moore property. This was employed in \cite{Cio}  to give a unified proof for all these above examples. 

Apart from those mentioned families of groups, there is no other known example of a locally compact group that has the Howe--Moore property. It is therefore legitimate to ask if the Howe--Moore property holds only for those mentioned simple algebraic groups over local fields and for their relatives coming from groups acting on trees.  As this question is more difficult to be studied in full generality, in this paper we simply restrict to Question~\ref{que::main_question} above, leaving the general case to be studied in an upcoming paper. 

To make things even more easier, a test case is proposed to be studied with respect to Question~\ref{que::main_question}, this is the universal group $U(F)$ introduced by Burger and Mozes~\cite[Section~3]{BM00a}. Recall, $U(F)$ is a closed subgroup of the full group of automorphisms $\Aut(T_d)$ of a $d$-regular tree $T_d$. In fact, as $U(F)$ is not a simple group, we study its simple subgroup $U(F)^{+}$ (for its definition and its properties see below). The main reason why the universal group $U(F)^{+}$ is suggested as the first example that should be studied towards answering Question~\ref{que::main_question} is the following. By Caprace and De Medts [CDM11, Prop. 4.1], we know $F \leq  \Sym(\{1, \cdots, d\})$ is primitive if and only if every proper open subgroup of $U(F)^{+}$ is compact. The latter condition that every proper open subgroup
is compact is well-known to be satisfied by any group having the Howe--Moore property; however, it is not sufficient, in general, to imply the Howe--Moore property. To conclude, we emphasize a second question that stems from Question~\ref{que::main_question}:

\begin{question}
\label{que::second_main_question}
Is it true the group $U(F)^{+}$ does not have the Howe--Moore property when $F$ is primitive but not $2$--transitive?
\end{question}

In this article we provide a positive answer to Question \ref{que::second_main_question}, by constructing a Hilbert space $\mathcal{H}$ endowed with a family of inner products, and a non-trivial unitary representation $(\pi, \mathcal{H})$ of $U(F)^+$, when $F$ is primitive but not $2$-transitive, that has some of its matrix coefficients not vanishing at infinity and does not admit non-zero $U(F)^{+}$-invariant vectors.

We emphasize that the theory of unitary representations for closed subgroups of $\Aut(T_d)$ that do not act $2$--transitively on the boundary of $T_d$ is very less developed.  Just to mention, parabolically induced unitary representations of the universal group $U(F)^+$  have all their matrix coefficients vanishing at infinity, as it was proven in \cite{Ci_a}. Or that the Hecke algebra of $U(F)^{+}$, when $F$ being primitive but not $2$--transitive, with respect to the maximal compact subgroup $K$ is infinitely generated and infinitely presented. On the contrary, when $F$ is just $2$-transitive that Hecke algebra of $U(F)^+$ is commutative, finitely generated admitting only one generator (see \cite{Ci_bb}).

\medskip
Let us state the main results of this article. The first theorem provides a characterization of the group $U(F)^+$, with $F$ being primitive, but not cyclic of prime order, to have the Howe--Moore property.  

\begin{theorem}
\label{thm::equiv_Howe-Moore_1}
Let $F \leq \Sym\{1,\cdots, d\}$ be primitive, but not cyclic of prime order. Then $U(F)^+$ has the Howe--Moore property if and only if for every unitary representation $(\pi, \mathcal{H})$ of $U(F)^+$, without non-zero $U(F)^+$-invariant vectors, and for every $v \neq 0 \in \mathcal{H}$ the closed isotropy subgroup $(U(F)^+)_{v}:=\{ g \in  U(F)^+ \; \vert \; \pi(g)(v)=v \}$ is compact.
\end{theorem}

As a second theorem we obtain:
\begin{theorem}
\label{thm::rel_Howe-Moore_1}
Let $F \leq \Sym\{1,\cdots, d\}$ be primitive, but not cyclic of prime order. Then for each $\xi \in \partial \T$ the group $U(F)^+$ has the relative Howe--Moore property with respect to its closed subgroup $(U(F)^+)_{\xi}^{0}:=\{g \in  U(F)^+\; \vert \; g(\xi)=\xi \text{ and } g \text{ is elliptic}\}$.
 \end{theorem}

The relative Howe--Moore property was introduced and studied by Cluckers--de Cornulier--Louvet--Tessera--Valette~\cite{CCL+} and it is a weakening of the Howe--Moore property.


\begin{definition}
Let $G$ be a locally compact group and $H$ be a closed subgroup of $G$. We say the pair $(G, H)$ has the \textbf{relative Howe--Moore property} if every unitary representation $(\pi, \mathcal{H})$ of $G$ either it has non-zero $H$--invariant vectors, or the restriction $\pi\vert_{H}$ is a unitary representation of $H$ that has all its matrix coefficients vanishing at infinity (with respect to $H$).
\end{definition}

Theorems~\ref{thm::equiv_Howe-Moore_1} and~\ref{thm::rel_Howe-Moore_1} are a consequence of the following strengthening of Mautner's phenomenon for locally compact groups acting on $d$-regular trees and having Tits' independence property (see Definition~\ref{def::Tits_prop}). We state this result in the general framework of groups acting on locally finite infinite trees.

\begin{proposition}
\label{prop::strong_Mautner_1}
Let $T$ be a locally finite infinite tree. Let $G \leq  \Aut(T)$ be a closed, non-compact subgroup with Tits' independence property and $(\pi,\mathcal{H})$ be a unitary representation of $G$. Assume there exist a sequence $\{g_n\}_{n>0} \subset G$ and a vector $v \in \mathcal{H}$ such that $g_n \to \infty$ and $\{\pi(g_n)(v)\}_{n>0}$ weakly converges to a non-zero vector $v_0 \in \mathcal{H}$.

Then, for each $\xi \in \bd T$, with the property that there exist $x \in T$ and a subsequence $\{g_{n_k}\}_{n_k>0}$ such that $g_{n_k}(x) \to \xi $, we have that $N^{+}_{\{g_{n_k}\}_{n_k}} \leq G_{\xi}^{0} \leq G_{v_0}$. Here $G_{\xi}^{0}:=\{ g \in G \; \vert \; g (\xi)=\xi \text{ and } g \text{ is elliptic} \}$ and $G_{v_0}:=\{g \in G \; \vert \; \pi(g)(v_0)=v_0 \}$.

Moreover, if $G_{v_0}$ is non-compact and does not contain hyperbolic elements then $G_{v_0}=G_{\eta}^{0}$, for an  end $\eta \in \bd T$.
\end{proposition}

As a corollary we obtain:

\begin{corollary}\label{cor::end_stab_vector_1}
Let $F$ be primitive but not $2$--transitive and not cyclic of prime order. Let $(\pi,\mathcal{H})$ be a unitary representation of $U(F)^+$ without non-zero $U(F)^+$-invariant vectors and where $\mathcal{H}$ is a separable complex Hilbert space. Consider $v \in \mathcal{H}$ such that $(U(F)^+)_{v}$ is not compact. Then $(U(F)^+)_{v}=(U(F)^+)_{\eta}^{0}$ for an end $\eta \in \bd \T$.
\end{corollary}

Finally we state the third main result of this article answering Question \ref{que::second_main_question}.
\begin{theorem}
\label{thm_iff_HM}
Let $F$ be primitive but not cyclic of prime order. Then $U(F)^+$ has the Howe--Moore property if and only if $F$ is $2$-transitive.
\end{theorem}

The structure of this article is the following. In the first place, Section~\ref{sec::new_prop} adds to the (long) list of Section~\ref{sec::def_prop} new equivalent conditions of the universal group $U(F)^+$ to act $2$--transitively on the boundary $\bd \T$. Section~\ref{sec::H-M-prop} gathers new algebraic and harmonic analytic properties of the group $U(F)^+$, when $F$ is primitive, but not cyclic of prime order. Those properties are used to prove Theorem~\ref{thm::equiv_Howe-Moore_1} and Proposition~\ref{prop::strong_Mautner_1}.  The second main result of this article, Theorems~\ref{thm::rel_Howe-Moore_1}, is proved in Section~\ref{sec::rel_Howe_Moore}. Finally, the most important result, Theorem \ref{thm_iff_HM}, is proven in Section \ref{sec::const_unit_rep}.

\subsection*{Acknowledgements} I would like to thank Pierre-Emmanuel Caprace for the valuable discussions during the author's PhD thesis  when the results of Sections \ref{sec::new_prop} and \ref{sec::H-M-prop} were proven. I am also grateful to Marc Burger and Alain Valette for useful discussions during those years, Uri Bader for his encouragements,  and the referees who did not accept for publication the partial results of Sections \ref{sec::new_prop}, \ref{sec::H-M-prop} and \ref{sec::rel_Howe_Moore}. Finally, many thanks to the manager Urs Nef for his encouragements during my interview at Leica Geosystems to break the question and come up with an innovative, easy solution.

\section{Definitions and some known properties}
\label{sec::def_prop}

First, let us fix some general notation. Denote by $\T$ a \textbf{$d$-regular tree, with $d \geq 3$,} and by $\Aut(\T)$ its full group of automorphisms that is endowed with the compact-open topology. Denote by $\dist_{\T}(\cdot, \cdot)$ the usual metric on $\T$ where every edge has length one.  Consider the vertices of each edge of $T_d$ are colored differently and each edge uses the same set of two colors. We say $g \in  \Aut(T_d)$ is \textbf{type-preserving} if $g$ preserves the above coloration of the tree. Then $G \leq \Aut(T_d)$ is type-preserving if all elements of G are type-preserving. Let $\Aut(\T)^{+}$ be the group of all type-preserving automorphisms of $\T$. Denote by $\partial \T$ the ends of $\T$, and we call $\partial \T$ the boundary of $\T$.  For each two points $x,y \in \T \cup \bd \T$, let $[x,y]$ be the unique geodesic between $x$ and $y$ in $\T \cup \bd \T$. 

For $G \leq \Aut(\T)$ and $x, y \in \T\cup \bd \T$ we define $$G_{[x,y]}:=\{ g \in G \; \vert \; g \text{ fixes pointwise the geodesic } [x,y]\}.$$ In particular, $G_{x}= \{g \in G \; \vert \; g(x)=x\}$. For $\xi \in \bd \T$ we define 
$$G_{\xi}:= \{g \in G \; \vert \; g(\xi)=\xi \} \; \;  \text{ and }  \; \;G_{\xi}^{0}:= \{g \in G \; \vert \; g(\xi)=\xi \text{ and } g \text{ fixes a vertex of } \T\}.$$ Note  $G_{\xi}$ can contain hyperbolic elements; if this is the case then $G_{\xi}^{0} \lneq G_{\xi}$.

For a hyperbolic element $\gamma \in \Aut(\T)$, we denote $\vert \gamma \vert:= \min_{x \in \T}\{ \dist_{\T}(x, \gamma(x))\}$, which is called the translation length of $\gamma$, and set $\Min(\gamma):=\{x \in \T \; \vert \;  \dist_{\T}(x, \gamma(x))=\vert \gamma \vert\}$.

\medskip

Recall now the family of universal groups $U(F)$ introduced by Burger--Mozes in \cite[Section~3]{BM00a}. In his PhD thesis \cite{Amann} Amann studies these groups from the point of view of their unitary representations.

\begin{definition}
\label{def::legal_color}
Let $\iota: E(T_d) \to \{1,...,d\} $, where $E(T_d)$ is the set of unoriented edges of the tree $T_d$.  The set $E(x) \subset E(T_d)$ consisting of all the edges containing the vertex $x \in T_d$ is called the star of $x$. We say the map $\iota$ is a \textbf{legal coloring} of the tree $T_d$ if for every vertex $x  \in T_d$ the map $\iota\vert_{E(x)}$ is in bijection with $\{1,...,d\} $. 
\end{definition}

\begin{definition}
\label{def::universal_group}
Let $F$ be a subgroup of permutations of the set $\{1,...,d\}$ and let $\iota$ be a legal coloring of $T_d$. We define the \textbf{universal group}, with respect to $F$ and $\iota$, to be 
$$ U(F):= \{g \in \Aut(T_d) \; \vert \; \iota \circ g \circ (\iota \vert_{E(x)})^{-1} \in F, \text{ for every x}\in T_d \}. $$
When $F$ is the full permutation group $\Sym(\{1,...,d\})$, $U(F)$ equals $\Aut(T_d)$. If $F=\id$ then $U(F)$ is the group of all automorphisms preserving the coloring $\iota$ of $T_d$. 

We say an element $g \in \Aut(T_d)$ is \textbf{edge-stabilizer} if there is an edge of $T_d$ stabilized be $g$ pointwise. 
By $U(F)^{+}$ we denote the subgroup of $U(F)$ generated by the edge-stabilizers in $U(F)$. When $F= \Sym(\{1,...,d\})$,  $U(F)^{+}=\Aut(T_d)^+ $ the group of type-preserving automorphisms of $T_d$,  the latter being a simple subgroup of index $2$ of $\Aut(\T)$ (see Tits~\cite{Ti70}).
\end{definition}

By  \cite[Prop.~52]{Amann} the groups $U(F),U(F)^{+} $ are independent of the legal coloring $\iota$ of $T_d$. From the definition, $U(F)$ and $U(F)^{+}$ are closed subgroups of $\Aut(T_d)$.


Another key property which is used in the sequel is the following.
 
\begin{definition}[See \cite{Ti70}]
\label{def::Tits_prop}
Let $T$ be a locally finite tree and let $G \leq \Aut(T)$ be a closed subgroup. We say  $G$ has \textbf{Tits' independence property} if for every edge $e$ of $T$ we have the equality $G_e=G_{T_{e,1}}G_{T_{e,2}}$, where $T_{e,i}$ are the two infinite half sub-trees of $T$ emanating from the edge $e$ and containing $e$,  and $G_{T_{e,i}}$ is the pointwise stabilizer of the half-tree $T_{e,i}$.
\end{definition}

We mention Tits' independence property guarantees the existence of enough rotations in the group $G$. It is used in the work of Tits as a sufficient condition to prove simplicity of large subgroups of $\Aut(\T)$ (see Tits~\cite{Ti70}). In his thesis \cite[Theorem~2]{Amann} Amann employes it to give a complete classification of all super-cuspidal representations of a closed subgroup in $\Aut(\T)$ acting transitively on the vertices and on the boundary of $\T$ and having Tits' independence property. For a closed subgroup $G$ of $\Aut(\T)$ that acts transitively on the vertices and on the boundary of $\T$ but which does not have Tits' independence property less is known about the complete classification of all its super-cuspidal representations. In contrast, for the above mentioned groups, with or without Tits' independence property, the remaining two classes of irreducible unitary representations, namely the special and respectively, the spherical ones,  are completely classified in Fig{\`a}-Talamanca--Nebbia~\cite[Chapter III, resp. Chapter II]{FigaNebbia}. 

\medskip
Following Burger--Mozes and Amann~\cite{BM00a, Amann}, we enumerate the following properties of the groups $U(F)$ and $U(F)^{+}$:
\begin{list}{\arabic{qcounter})~}{\usecounter{qcounter}}
\item
\label{diff_properties_1}
$U(F)$ and $U(F)^{+}$ have Tits' independence property; 
\item
\label{diff_properties_4_bis}
$U(F)^{+}$ is trivial or simple;
\item
\label{diff_properties_2}
if $F$ acts transitively on the set $\{1,...,d\}$ then $U(F)$ acts transitively on the edges of $\T$ and  $U(F), U(F)^{+}$ are unimodular groups; 
\item
\label{diff_properties_3}
$U(F)$ (resp., $U(F)^{+}$) acts $2$--transitively on the boundary $\partial \T$ if and only if $F$ is $2$--transitive;
\item
\label{diff_properties_4}
When $F$ is transitive and generated by its point stabilizers, by \cite[Prop.~57]{Amann}  $U(F)^{+}$ is edge-transitive, $U(F)^{+}= U(F) \cap \Aut(\T)^{+}$ and $U(F)^{+}$ is of index $2$ in $U(F)$. For example, this is the case when $F$ is primitive but not cyclic of prime order. When $F$ is primitive and not generated by its point stabilizers, all point stabilizers of $F$ are equal. This implies all point stabilizers of $F$ are just the identity, that $F$ is cyclic of prime order and the group $U(F)^+$ is trivial. Under these hypotheses, $U(F)$ is a nontrivial, discrete subgroup of $\Aut(\T)$. In order to avoid heavy formulation, we simply use \textit{$F$ is primitive} to mean that \textit{$F$ is primitive but not cyclic of prime order}.
\end{list}

Regarding the $2$--transitive action on the boundary $\partial \T$ of a group $G \leq \Aut(\T)$, we have the following general well-known equivalences.

\begin{proposition}(See \cite{FigaNebbia}, \cite[Lemma~3.1.1]{BM00a})
\label{prop::transitivity1}
For a closed subgroup $G \leq \Aut(\T)$ the following are equivalent:
\begin{list}{\roman{qcounter})~}{\usecounter{qcounter}}
\item
$G_{x}$ is transitive on $\partial \T$, for each $x \in \T$;
\item
$G$ is non-compact and there exists $x \in \T$ such that $G_{x}$ acts transitively on $\partial \T$;
\item
$G$ is non-compact and transitive on $\partial \T$;
\item
$G$ is 2-transitive on $\partial \T$.
\end{list}
\end{proposition}

\begin{remark}(See \cite[Prop.~10.2]{FigaNebbia})
One can prove if a closed, non-compact subgroup $G \leq \Aut(\T)$ is transitive on the boundary $\bd \T$, then either $G$ is transitive on the vertices of $\T$, or it has exactly two $G$-vertex-orbits in $\T$.
\end{remark}

The following proposition gives us a slightly different, but equivalent, description of a $2$--transitive action on $ \partial \T$.

\begin{proposition}(See \cite{CaCi, Cio_thesis})
\label{prop::transitivity2}
Let $G$ be a closed, non-compact subgroup of $\Aut(\T)^+$. The following are equivalent:
\begin{list}{\roman{qcounter})~}{\usecounter{qcounter}}
\item
$G$ is strongly transitive on $\T$, where $\T$ is viewed as a $1$-dimensional Euclidean building;
\item
$G$ is $2$--transitive on $\partial \T$;
\item
$G$ has no fixed point in $\partial \T$ and for some point $\xi \in \partial \T$, the stabilizer $G_\xi$ acts co-compactly on $\T$;
\item
$G$ acts co-compactly on $\T$, without a fixed point in $\partial \T$, and there exists a compact open subgroup in $G$ having finitely many orbits on $\partial T$;
\item
$G$ acts co-compactly on $\T$, without a fixed point in $\partial \T$, and every compact open subgroup has finitely many orbits on $\partial T$;
\item
$G$ acts co-compactly on $\T$ and $(G,G_x)$ is a Gelfand pair, where $x$ is a vertex of $\T$.
\end{list}
\end{proposition}

\begin{convention}
\label{not::general_notation}
Let $F$ be primitive and fix a coloring $\iota$ of $\T$. From now on and to simplify the notation, we set $\G:= U(F)^+$.
\end{convention}

When $F$ is primitive but not $2$--transitive, the universal group $\G$ still enjoys some of the properties of closed, non-compact subgroups of $\Aut(\T)$ that act $2$--transitively on the boundary $\bd \T$. Recall  \cite[Remark 2.1, Lemma 2.2]{Ci_a}.

\begin{remark}
\label{rem::existence_hyp_elements}
As $F$ is primitive, given an edge $e' \in E(T_d) $ at odd distance from $e$, one can construct, using the definition of $\G$, a hyperbolic element in $\G$ translating $e$ to $e'$. Moreover, every hyperbolic element in $\G$ has even translation length, as $\G$ has only type-preserving elements.
\end{remark}
 
For a vertex $x \in T_d$ and en edge $e$ in the start of $x$, set $T_{x,e}$ be the half-tree of $T_d$ emanating from the vertex $x$ and containing the edge $e$.

\begin{lemma}($KA^{+}K$ decomposition)
\label{lem::KAK_decomposition} Let $x \in T_d$ and $e$ an edge in the start of $x$. Let $F$ be primitive and $K:= \G_x$. Then $\G$ admits a $KA^{+}K$ decomposition, where $$A^{+}:=\{\gamma \in \G \; \vert \; \gamma \text{ hyperbolic, }  e \subset \Min(\gamma), \; \gamma(e) \subset T_{x,e} \} \cup \{\id\}.$$
\end{lemma}

\begin{proof}
Let $g \in \G$. If $g(x)=x$, then $g \in K$. If not, then $g(x) \neq x$. Consider the geodesic segment $[x, g(x)]$ in $T_d$ and denote by $e_1$ the edge of the star of $x$ that belongs to $[x, g(x)]$. Note $[x, g(x)]$ has even length and there exists $k \in K$ such that $k(e_1)=e$; therefore, $kg(x) \in T_{x,e}$. Then, by Remark~\ref{rem::existence_hyp_elements}, there is a hyperbolic element  $\gamma \in \G$ of translation length equal to the length of $[x, g(x)]$, with $e \subset \Min(\gamma)$, that translates the edge $e$ inside $T_{x,e}$ and such that $\gamma(x)=kg(x)$; thus $\gamma^{-1}kg \in K$. Note, the $KA^{+}K$ decomposition of an element $g \in \G$ is not unique.
\end{proof}

Note, when $F$ is $2$-transitive, the $K$-double quotient $ K \backslash A^{+} / K$ equals $\{a^n \; \vert \; n \geq 0\}$, where $a \in \G$ is a hyperbolic element of translation length $2$, with $e \subset \Min(\gamma)$, and having $a(e)$ in $T_{x,e}$. If $F$ is not $2$-transitive, the set $K$-double quotient $ K \backslash A^{+} / K$ is a very big non-abelian sub-semi-group.



\section{New characterizations of $2$--transitivity}
\label{sec::new_prop}

In addition to property \ref{diff_properties_3}) from Section \ref{sec::def_prop}, the next propositions provide new  characterizations of the $2$--transitivity of the $\G$-action on the boundary $\bd \T$. These characterizations are expressed  in terms of the subgroups $\G_\xi$ and the existence of hyperbolic elements in $\G_\xi$, where $\xi \in \bd \T$.

\begin{proposition}
\label{lem::hyp_ends}
$\G$ is transitive on the boundary $\bd \T$ if and only if for every $\xi \in \bd \T$, $\G_{\xi}$ contains a hyperbolic element.
\end{proposition}

\begin{proof}
`$\Rightarrow$'  If  $\G$ is transitive on the boundary $\bd \T$ then, by \cite[Prop. 3.4]{Ti70}, $\G$ admits hyperbolic elements; so the implication is easy and follows from Propositions~\ref{prop::transitivity1} and~\ref{prop::transitivity2}.

`$\Leftarrow$' Suppose, for every $\xi \in \bd \T$ the group $\G_{\xi}$ contains a hyperbolic element. We want to show  $F$ is $2$--transitive, which by property \ref{diff_properties_3}) from Section \ref{sec::def_prop} is equivalent to the transitivity of $\G$ on the boundary $\bd \T$.

Assume  $F$ is not $2$--transitive. Therefore, there exists a color, say $1$, such that $F_{1}$ is not transitive on $\{ 2,...,d \}$ (otherwise $F$ would be $2$--transitive). Moreover, without loss of generality, we can suppose  $F_{1}$ does not contain any element mapping $2$ to $3$.  

Consider a vertex $v \in \T$ and starting from this vertex we want to construct a particular geodesic ray $[v, \eta )$, with $\eta \in \bd \T$. We enumerate its vertices by $\{v_{i}\}_{i \geq 0}$, such that $v_{0}=v$. We construct the geodesic ray $[v, \eta )$ by choosing the vertices $\{v_i\}_{i \geq 0}$ such that the successive colors of the edges along the path $[v, \eta )$ are: $$(12)(13)(12)^{2}(13)(12)^{3}(13)\cdots (12)^{m}(13)(12)^{m+1}(13)\cdots,$$where $(12)^{m}$ denotes the geodesic segment of length $2m$ colored successively with the pair of ordered colors $(1,2)$. Note,  the choice of the vertices $\{v_i\}_{i \geq 0}$ with the above prescribed coloration is unique. Moreover, it determines a unique end $\eta \in \bd \T$ and thus, a unique geodesic ray $[v, \eta )$.

By hypothesis, $\G_{\eta}$ contains at least one hyperbolic element. Take one of those and denote it by $b$; consider  its translation length is $2N$ and that its attracting end is $\eta$. Notice there exists a vertex $v_{b}$ on the geodesic ray $[v,\eta)$ such that for every $v \in [v_b,\eta)$ we have $b(v) \in [v_b,\eta)$. Moreover, if we travel deep enough on the geodesic ray $[v_b,\eta)$ we find a geodesic segment whose coloring is the following string $(12)^{N}(13)(12)^{N}$. Let $x$ be the midpoint of this geodesic segment. The two edges in $[v_b,\eta)$ emanating from $x$ have colors $1$ and $3$, where the two edges in $[v_b,\eta)$ emanating from $b(x)$ have colors $1$ and $2$. We obtained  $b$ sends the pair of ordered colors $(1,2)$ into the pair of ordered colors $(1,3)$. This contradicts  $F_{1}$ does not contain any element mapping $2$ to $3$. The proposition stands proven.
\end{proof}

\begin{proposition}
\label{lem::second_trans_hyp}
Let $F$ be transitive and generated by its points stabilizers. Then $\G$ acts transitively on the boundary $\bd \T$ if and only if every hyperbolic element of $\G$ is a power of a hyperbolic element of translation length two.
\end{proposition}

\begin{proof}
The implication `$\Rightarrow$' follows by applying Propositions~\ref{prop::transitivity1} and~\ref{prop::transitivity2}.

Consider the reverse implication `$\Leftarrow$'. Suppose moreover that $\G$ is not transitive on the boundary $\bd \T$. This is equivalent to $F$ not being $2$--transitive. Therefore, there exists a color, say $1$, such that $F_1$ is not transitive on $\{2,\cdots, d\}$. Moreover, without loss of generality, we can assume  $F$ does not contain any element sending the pair of ordered colors $(1,2)$ to the pair of ordered colors $(1,3)$. 

Construct in $\T$ a bi-infinite geodesic line colored using only the array of ordered colors $(1,2,1,3)$ and denote it by $\ell$. As $F$ is transitive and generated by its points stabilizers, the group $\G$ is edge-transitive on $\T$; therefore, $\G$ acts co-compactly on $\T$. Apply then \cite[Proposition~2.9]{CaCi} to the geodesic line $\ell$. There exists thus a hyperbolic element $h \in \G$ such that its translation axis intersects $\ell$ into a geodesic segment of a big length. By hypothesis $h$ is a power of a hyperbolic element of translation length $2$. This implies  there exists an element  in $\G$ such that the pair of ordered colors $(1,2)$ is sent into the pair of ordered colors $(1,3)$. This gives an element of $F$ sending $(1,2)$ to $(1,3)$ which is a contradiction. The conclusion follows. 
\end{proof}

\section{Around the Howe--Moore property}
\label{sec::H-M-prop}

Recall the following definition.

\begin{definition}
\label{def::U_alpha}
Let $G$ be a locally compact group and let $\alpha=\{g_{n}\}_{n>0}$ be a sequence of elements in $G$. Corresponding to $\alpha$ we define the set:
\[ 
U^{+}_{\alpha}:= \{ g \in G \; | \; \lim\limits_{n \to \infty} g_{n}^{-1}gg_{n}=e \}. 
\]

Note $U^{+}_{\alpha}$ is a subgroup of $G$. It is called the \textbf{contraction group} corresponding to $\alpha$. Because it does not need to be closed in general, we denote by $N^{+}_{\alpha}$ the closed subgroup $\overline{U^{+}_{\alpha}}$. In the same way, but using $g_{n}gg_{n}^{-1}$ we define $U^{-}_{\alpha}$ and $N^{-}_{\alpha}$.

When $g_n=a^{n}$, for every $n >0$, and for some $a \in G$, we simply denote $U^{+}_{\alpha}$ by $U^{+}_{a}$.
\end{definition} 

For example, the next easy lemma describes the contraction group of a sequence $\alpha=\{g_n\}_{n >0} \subset \Aut(T)$ that enjoys particular properties. 

\begin{lemma}
\label{lem::contr_group}
Let $T$ be a locally finite infinite tree and $\alpha=\{g_n\}_{n >0} \subset G \leq  \Aut(T)$. Assume there exist $x \in T$ and $\xi \in \bd T$ with the property that $g_n(x) \to \xi$ when $n \to \infty$, the limit being considered with respect to the cone topology on $T \cup \bd T$. Then every element $g \in G \leq \Aut(T)$ that fixes pointwise the $x$-rooted half-tree of $T$ containing $\xi$ in its boundary, is an element of $U^{+}_{\alpha}$.
\end{lemma}

\begin{proof}
Let $g  \in G \leq \Aut(T)$ fixing pointwise the $x$-rooted half-tree $T_{x,\xi} \subset T$ that contains $\xi$ in its boundary. We have to prove  $\lim\limits_{n \to \infty} g_{n}^{-1}gg_{n}=e$. This is equivalent to showing that for every $r >0$, there exists $N_r >0$ such that $g_{n}^{-1}gg_{n}(B(x, r))=B(x,r)$ pointwise, for every $n >N_r$, where $B(x, r) \subset T$ is the ball centered in $x$ and of radius $r$. The latter condition is translated as $g(B(g_n(x), r))= B(g_n(x), r)$ pointwise, for every $n > N_r$. This equality is true by hypothesis, as $g_n(x) \to \xi$ and $g$ fixes pointwise the half-tree $T_{x, \xi}$.  The conclusion follows.
\end{proof}

Note, the sequence $\{g_n=a^{n}\}_{n>0}$, where $a \in \Aut(T)$ is a hyperbolic element, satisfies the hypotheses of Lemma~\ref{lem::contr_group};  the point $\xi \in \bd T$ is the attracting end of $a$ and $x$ a vertex of the translation axis of $a$. In this particular case, one can prove more when considering the group $\G$. The next lemma emphasizes the importance of contractions groups.
 
\begin{lemma}
\label{lem::hyp_C_0}
Let $F$ be primitive. Let $a$ be a hyperbolic element in $\G$. Then $\G= \left\langle U^{+}_{a},U^{-}_{a}\right\rangle$. 
In particular, for any unitary representation $(\pi, \mathcal{H})$ of $\G$, without non-zero $\G$-invariant vectors and where $\mathcal{H}$ is a separable complex Hilbert space, and for any $(v,w)$-matrix coefficient of $(\pi, \mathcal{H})$, with $v,w \in \mathcal{H}$, we have  $\lim\limits_{n \to \infty} \vert c_{v,w}(a^n) \vert =0$.

\end{lemma}

\begin{proof}
Denote  by $\ell$ the translation axis of the hyperbolic element $a$ and its repelling and attracting unique ends by $\xi_{-}$ and respectively, $\xi_{+}$.

Let $g$ be an element in $U^{+}_{a}$, which by definition has the property $\lim\limits_{n \to \infty} a^{-n}ga^{n}=e$. One deduces $g(\xi_{+})=\xi_+$ and that $g$ fixes at least one vertex of $\T$. Therefore $g \in \G^{0}_{\xi_+}:=\{ h \in \G \; \vert \; h (\xi_+)=\xi_+ \text{ and } h \text{ is elliptic} \} $. We conclude that $U^{+}_{a} \leq \G^{0}_{\xi_+}$. In the same way  $U^{-}_{a} \leq \G^{0}_{\xi_-}$.

Moreover, we claim  $\G^{0}_{\xi_+} \leq  \left\langle U^{+}_{a},U^{-}_{a}\right\rangle $. Indeed, let $g \in \G^{0}_{\xi_+}$ and consider a vertex $x \in \ell$ fixed by $g$. The geodesic ray $[x,\xi_+) \subset \ell$ is thus pointwise fixed by $g$. Let $e$ be an edge in the  interior of  $[x,\xi_+)$. By Tits' independence property we obtain  $g=g_1h_1$, where $g_1$ fixes pointwise the half-tree emanating from the edge $e$ and containing the end $\xi_+$ and $h_1$ fixes pointwise the other half-tree emanating from $e$, which moreover contains the end $\xi_-$. By Lemma~\ref{lem::contr_group} $h_1 \in U^{-}_{a}$ and $g_1 \in U^{+}_{a}$. In particular $\G_{e} \leq  \left\langle U^{+}_{a},U^{-}_{a}\right\rangle$ and our claim follows.  

To conclude, we obtain  $ \G^{0}_{\xi_+}, \G^{0}_{\xi_-} \leq \left\langle U^{+}_{a},U^{-}_{a}\right\rangle$  and $\left\langle U^{+}_{a},U^{-}_{a}\right\rangle$ is an open subgroup of $\G$, because it contains the open subgroup $\G_{e}$. But $\left\langle U^{+}_{a},U^{-}_{a}\right\rangle$ is not compact as otherwise it would be contained in a vertex stabilizer by \cite[Lemma~2.6]{CaMe11}, which is impossible (see also Lemma~\ref{lem::Stab_end_notcompact}). Thus $ \left\langle U^{+}_{a},U^{-}_{a}\right\rangle $ is an open, non-compact subgroup of $\G$. By Caprace--De Medts~\cite[Prop.~4.1]{CaMe11}  the subgroup $F$ is primitive if and only if every proper open subgroup of $\G$ is compact, thus we must have  $ \left\langle U^{+}_{a},U^{-}_{a}\right\rangle =\G$. 

The last part of the lemma follows by applying \cite[Lemma~2.9 and Lemma~3.1]{Cio}.
\end{proof}

In a similar manner as in the proof of Lemma \ref{lem::hyp_C_0}, one can show $\G= \left\langle U^{+}_{a},U^{-}_{b}\right\rangle$, where $a,b \in \G$ are two hyperbolic elements with the attraction end of $a$ different than the repelling end of $b$.

Lemma~\ref{lem::hyp_C_0} motivates the following remark.

\begin{remark}
\label{rem::stab_vect}
Let  $F$ be primitive but not $2-$transitive. Let $(\pi,\mathcal{H})$ be a unitary representation of $\G$ without non-zero $\G$-invariant vectors and where $\mathcal{H}$ is a separable complex Hilbert space. Consider $v \neq 0 \in \mathcal{H}$ and let $\G_{v}:=\{ g \in \G \; \vert \; \pi (g)(v)=v \}$. We record the following two possibilities for the closed subgroup $\G_v$ of $\G$:
\begin{list}{\roman{qcounter})~}{\usecounter{qcounter}}
\item
$\G_{v}$ is compact;
\item
$\G_{v}$ is not compact. In this case, as the unitary representation $(\pi, \mathcal{H})$ is without non-zero $\G$-invariant vectors, by Lemma~\ref{lem::hyp_C_0} $\G_{v}$ does not contain hyperbolic elements. Applying the result of Tits \cite[Prop.~3.4]{Ti70} we conclude $\G_{v}$ is a closed subgroup of the pointwise-stabilizer of an end $\xi \in \bd \T$.  Note, this end $\xi$ is unique. Indeed, if there were at least two different ends $\xi_1, \xi_2 \in \bd \T$ that are stabilized by $\G_{v}$ then $\G_{v}$ would fix a point in the tree $\T$ and so $\G_{v}$ would be compact. 
\end{list}
\end{remark}

Remark \ref{rem::stab_vect} is completed by the following lemma and proposition.

\begin{lemma}
\label{lem::Stab_end_notcompact}
Let $F$ be primitive. Then for every $\xi \in \partial \T$ the subgroup $\G^{0}_{\xi}$ is not compact, where $\G^{0}_{\xi}:=\{ g \in \G \; \vert \; g (\xi)=\xi \text{ and } g \text{ is elliptic} \}$. 
\end{lemma}

\begin{proof}
Suppose there is $\xi \in \bd \T$ such that $\G^{0}_{\xi}$ is compact. Using \cite[Lemma~2.6]{CaMe11} there exists a vertex $x \in \T$ such that $\G^{0}_{\xi} \leq  \G_{x}$. Take the geodesic ray $\left[ x, \xi \right)$ and let $e$ be an edge in $\left[ x, \xi \right)$ with $\dist_{\T}(e,x)>1$. Using Tits' independence property of $\G$ we claim for every $g \in \G_{e}$, $g(e')=e'$, where $e'$ is an edge contained in $\left[ x, \xi \right)$ such that $d(e', x) \leq d(e,x)-1$. Indeed, we have $\G_{e}=\G_{T_{e,1}}\G_{T_{e,2}}$, where $T_{e,1}$ and $T_{e,2}$ are the two infinite, disjoint, half-trees of $\T$ rooted, respectively, at the two vertices of $e$ and containing $e$. Therefore, if $g \in \G_{T_{e,1}}$ then $g(e')=e'$, where $T_{e,1}$ is the half-tree not containing the end $\xi$. If $g \in \G_{T_{e,2}}$ then $g \in \G_{\xi} \leq \G_v$ and thus $g(e')=e'$.

We conclude  $\G_{e} \subset \G_{e'}$, for any two edges $e,e'$ of $\left[ x, \xi \right)$ with $\dist_{\T}(e,x)>1$ and $d(e', x) \leq d(e,x)-1$. In particular, if we take $e'$ such that $d(e', x)= d(e,x)-1$ then we obtain $F_{i(e)} \leq F_{i(e')}$, where by $F_{j}$, with $j \in \left\lbrace 1,...,d \right\rbrace$,  we denote the stabilizer of the color $j$ in the permutation subgroup $F$.

Now, as the geodesic ray $\left[ x, \xi \right)$ is $d$-colored, there are at least two different colors on this geodesic ray, say $k\neq j $, for which $F_{j}=F_{k}$.

Choose an edge $e_{1}$ on the axis $[ x, \xi )$ which is colored $j$. Consider the unique bi-infinite geodesic line $\ell$ containing $e_1$ and colored only with $j$ and $k$, alternatively. Now $\G_{\ell}:=\{ g \in \G \; \vert \; g \text{ stabilizes } \ell \}$ contains $\G_{e_{1}}$. This is because if the color $j$ is fixed then also is the color $k$, as $F_{j}=F_{k}$, and vice versa. Therefore, $\G_{\ell}$ is an open subgroup in $\G$.

But $\G_{\ell}$ is not compact because it contains a hyperbolic element of translation length $2$, i.e., whose translation axis is $\ell$. Indeed, observe  the group $U(\id)$ is a subgroup of $U(F)$ and it preserves the fixed coloring $i$ of the tree. In particular, we can construct in $U(\id)$ a hyperbolic element $h$ of translation length $2$ whose translation axis is $\ell$. Moreover, this hyperbolic element $h$ belongs to $\Aut(\T)^{+}$.  As $F$ is primitive, apply \cite[Prop.~57]{Amann}. We obtain  $U(F)^{+}= U(F) \cap \Aut(\T)^{+}$ and so $h \in \G_{\ell} \leq U(F)^{+}$. We conclude  $\G_{\ell}$ is a non-compact, open, proper subgroup of $\G$ contradicting~\cite[Prop.~4.1]{CaMe11}. The conclusion follows.
\end{proof}

\begin{remark}
\label{rem::conv_endpoint}
Recall,  a locally finite tree $T$ is endowed with the cone topology which turns $T \cup \bd T$ into a compact topological space. Let $x \in T$ and take a sequence $\{g_n\}_{n>0} \subset \Aut(T)$ such that $g_n \to \infty$. Therefore $\dist_{T}(x, g_n(x)) \to \infty$. By compactness, we can extract a subsequence $\{g_{n_k}\}_{n_k}$ such that $\{g_{n_k}(x)\}_{n_k}$ converges with respect to the cone topology to an end $\xi \in \bd T$. 
\end{remark}

\begin{proof}[Proof of Proposition \ref{prop::strong_Mautner_1}]
First of all, by Mautner's phenomenon \cite[Ch.~III, Thm.~1.4]{BM} (or \cite{Cio}) the vector $v_{0}$ is  $N^{+}_{\{g_{n}\}_{n}}$-invariant. So $N^{+}_{\{g_{n}\}_{n}} < G_{v_{0}}$.

Consider now a vertex $ x \in T$. As $g_{n} \to \infty$, by Remark \ref{rem::conv_endpoint} we extract a subsequence $\{g_{n_{k}}\}_{n_k>0} \subset \{g_{n}\}_{n>0}$ such that $g_{n_{k}} (x) \to \xi $ in the cone topology of $T \cup \bd T$. 

\medskip
For such an $\{g_{n_k}\}_{n_k>0}$ and $\xi$, we claim  $N^{+}_{\{g_{n_k}\}_{n_k}} < G_{\xi}^{0}$. As $G_{\xi}^{0}$ is a closed subgroup, it is enough to show  $U^{+}_{\{g_{n_k}\}_{n_k}} < G_{\xi}^{0}$. Indeed, let $g \in U^{+}_{\{g_{n_k}\}_{n_k}}$. By definition we have $\lim\limits_{n \to \infty}g_{n_k}^{-1}g g_{n_k}=e$. As $g_{n_k}(x) \to \xi $ in the cone topology, we deduce  $g$ fixes the vertices $g_{n_k}(x)$ and so also the end $\xi$. We conclude  $g \in G_{\xi}^{0}$.

Moreover, we claim  $G_{\xi}^{0} \leq G_{v_{0}}$.  

Indeed, let $g \in G_{ \xi }^{0}$ and we want to prove  $g \in G_{v_{0}}$. Let $ [ x_{0}, \xi )$ be a geodesic ray fixed by $g$ and we index its vertices by $x_{j}$. Take an edge $ [ x_{j}, x_{j+1}]$, with $j \geq 0$ and we have  $g \in G_{ [ x_{j}, x_{j+1} ] }:=\{ h \in G \; \vert \; h [ x_{j}, x_{j+1}]= [ x_{j}, x_{j+1}] \}$. By Tits' independence property $g=t_{j}h_{j}$, where $t_{j} \in G_{T_{x_{j}}}$, $h_{j} \in G_{T_{x_{j+1}}}$, $T_{x_{j}}$ is the half-tree rooted at $x_{j}$ and containing the end $\xi$, and $T_{x_{j+1}}$ is the half-tree rooted at $x_{j+1}$ and not containing the end $\xi$. As $g$ and $t_{j}$ are stabilizing the end $\xi$ we deduce  $h_{j} \in G_{\xi}^{0}$. 

In addition, as $t_{j}$ fixes pointwise the half-tree $T_{x_j}$ containing the end $\xi$ and $g_{n_{k}} (x) \to \xi $, by Lemma \ref{lem::contr_group}  $\lim\limits_{n \to \infty} g_{n_k}^{-1}t_{j}g_{n_k}=e$. Thus $t_{j} \in N^{+}_{\{g_{n_k}\}_{n_k}} < G_{v_{0}}$. 

Remark  $h_{j} \to e$ when $j\to \infty$. Therefore $g h_{k}^{-1}=t_{k} \to g$. Because $G_{v_{0}}$ is a closed subgroup and $t_{k} \in G_{v_{0}}$ we deduce  $g \in G_{v_{0}}$. The claim follows.

\medskip
Let us now prove the last assertion of the proposition by supposing  $G_{v_0}$ is non-compact and does not contain hyperbolic elements. By the result of Tits \cite[Prop.~3.4]{Ti70}  $G_{v_0}$ is contained in the stabilizer $G_{\eta}^{0}$ of an end $\eta \in \bd T$. As $G_{v_0}$ is non-compact, choose a sequence $\{ h_m\}_{m>0} \subset G_{v_0}$ such that $h_m \to  \infty$. Remark  $\{\pi(h_m)(v_0)\}_{m>0}$ weakly converge to $v_0$ and for any vertex $x \in T$ we have  $h_m(x) \to \eta $, with respect to the cone topology of $T \cup \bd T$. Applying the first part of the proposition  $G_{\eta}^{0} \leq G_{v_0}$ and the conclusion follows.
\end{proof}

By Remarks~\ref{rem::stab_vect}  and~\ref{rem::conv_endpoint} we obtain Corollary \ref{cor::end_stab_vector_1}.

Using the results obtained until now, we are now ready to prove the equivalent characterization for the group $\G$ to have the Howe--Moore property given by Theorem \ref{thm::equiv_Howe-Moore_1} from the Introduction. 

\begin{proof}[Proof of Theorem \ref{thm::equiv_Howe-Moore_1}]
Remark, by \cite[Prop.~3.2]{CCL+} and~\cite[Prop.~4.1]{CaMe11}, the primitivity of $F$ is a necessary condition for the Howe--Moore property. Also, by~\cite[Remark~2.5 and Lemma~2.4]{Cio}, it is enough to consider only unitary representations over separable Hilbert spaces.

Moreover, the implication ``$\G$ has Howe--Moore property implies $\G_{v}$ is compact'' is easy and holds in general.  Indeed, if $\G_{v}$ is non-compact then there exists a sequence $\{g_{n}\}_{n} \subset \G_v$ with $g_n \to \infty$. Therefore, the corresponding matrix coefficients $\left\langle \pi(g_{n})(v),v \right\rangle=1$ which is a contradiction.

Consider now the reverse implication.

Suppose there exists a unitary representation $(\pi, \mathcal{H})$ of $\G$, without non-zero $\G$-invariant vectors and where $\mathcal{H}$ is a separable Hilbert space, two non-zero vectors $v, w \in \mathcal{H}$ and a sequence $\{g_{n}\}_{n>0} \subset \G$ such that $g_{n} \to \infty$  and $ \vert \left\langle \pi(g_{n}v, w \right\rangle \vert \nrightarrow 0$. Restricting to a subsequence and without loss of generality we can consider also that $ \{ \pi(g_{n})(v) \}_{n>0}$ weakly converges to $ v_{0} \neq 0 \in \mathcal{H}$. 

Take a vertex $ x \in \T$. As $g_{n} \to \infty$, by Remark \ref{rem::conv_endpoint} we extract a subsequence $\{g_{n_{k}}\}_{n_k>0} \subset \{g_{n}\}_{n>0}$ such that $g_{n_{k}} (x) \to \xi $ in the cone topology of $T \cup \bd T$.  Applying Proposition~\ref{prop::strong_Mautner_1} we obtain  $\G_{\xi}^{0} \leq \G_{v_0}$. By Lemma~\ref{lem::Stab_end_notcompact} we conclude  $\G_{v_0}$ is not compact, which contradicts our hypotheses. The theorem stands proven. 
\end{proof}

\section{The relative Howe--Moore property}
\label{sec::rel_Howe_Moore}

\begin{proof}[Proof of Theorem  \ref{thm::rel_Howe-Moore_1}]
Let $\xi \in \bd \T$ and let $(\pi, \mathcal{H})$ be a unitary representation of $G$ that does not have non-zero $\G_{\xi}^{0}$--invariant vectors. We need to prove  $\pi\vert_{\G_{\xi}^{0}}$ has all its matrix coefficients vanishing at infinity with respect to $\G_{\xi}^{0}$. Suppose this is not the case.  As $\G_{\xi}^{0}$ is not compact by Lemma~\ref{lem::Stab_end_notcompact}, there exists a sequence $\{g_n\}_{n\geq 0} \subset \G_{\xi}^{0}$, with $g_n \to \infty$, and $v,w \in \mathcal{H} \setminus \{0\}$ such that $\vert \left\langle \pi(g_{n}v, w \right\rangle \vert \nrightarrow 0$ when $g_n \to \infty$. Because the set $\{ \pi(g_{n}v\}_{n \geq 0}$ is bounded in the norm of $\mathcal{H}$, there exists $v_0 \in \mathcal{H}$ and a subsequence $\{n_k\}_{k \geq 0}$ such that $\{\pi(g_{n_k}v\}_{k \geq 0}$ weakly converges to $v_0$. In particular $\lim\limits_{n_k \to \infty}  \left\langle \pi(g_{n_k}v, w \right\rangle = \left\langle v_0, w \right\rangle$, thus by the assumption above  $v_0$ is a non-zero vector. Moreover,  $g_{n_k} \to \infty$ and for a fixed vertex $x \in \T$ we also have $g_{n_k}(x) \to \xi$ with respect to the cone topology on $\T$, when $n_k \to \infty$. The latter assertion follows because $\{g_{n_k}\}_{k\geq 0} \subset \G_{\xi}^{0}$ and because $g_{n_k} \to \infty$. As the hypotheses of Proposition~\ref{prop::strong_Mautner_1} are satisfied for the sequence $\{g_{n_k}\}_{k\geq 0} \subset \G_{\xi}^{0} \leq \G$, the fixed vertex $x \in \T$ and the vector $v, v_0$, we conclude  $\G_{\xi}^{0} \leq \G_{v_0}$. Thus $v_0$ is a non-zero vector of $\mathcal{H}$ that is invariant under $\G_{\xi}^{0}$, contradicting the assumption that $(\pi, \mathcal{H})$ does not have non-zero $\G_{\xi}^{0}$--invariant vectors. The conclusion follows.
\end{proof}

\section{The Howe--Moore property implies $F$ is $2$-transitive}
\label{sec::const_unit_rep}

Let us consider $F$ to be primitive but not $2$-transitive. The goal of the section is to construct a non-trivial unitary representation of $\G:= U(F)^{+}$ that does not have non-zero $\G$-invariant vectors but has matrix coefficients that do not vanish at infinity. This will confirm that $\G$ does not have the Howe--Moore property and prove Theorem \ref{thm_iff_HM}.  

The two key ingredients are the following. The first one is the result of Proposition \ref{lem::hyp_ends} saying that when $F$ is primitive but not $2$-transitive there are non-hyperbolic ends $\xi \in \partial T_d$ (i.e.,  $\G_{\xi}= \G_\xi^{0}$, thus there is no hyperbolic element in $\G$ stabilizing $\xi$). The second one is  Corollary \ref{cor::end_stab_vector_1} saying for a unitary representation of $\G$ with non-vanishing matrix coefficients the $\G$-stabilizers of its vectors are either compact in $\G$, or of the form $\G_{\xi}^0$, for some $\xi \in \partial T_d$. We construct a non-trivial unitary representation of $\G$, without  non-zero $\G$-invariant vectors, and having a vector whose $\G$-stabilizer is $\G_{\xi}^0$, for a non-hyperbolic end $\xi \in \partial T_d$. 

For the rest of the section we fix a non-hyperbolic end $\xi \in \partial T_d$, thus having  $\G_{\xi}= \G_\xi^{0}$.

\subsection{Half sub-trees and additive $\CC$-pseudo-measures on $\mathcal{HT}$}

Recall, a \textbf{half sub-tree} of $T_d$ associated with an oriented edge $[v,w] \in E(T_d)$  is the unique maximal (for the inclusion) sub-tree in $T_d$, denoted by $T_{[v,w]}$, that contains the edge $[v, w]$, but does not contain any neighbors of $v$ other than $w$. We denote 
\begin{equation*}
\begin{split}
 \mathcal{HT} := \{ S \subseteq  T_d  \;  \vert \; & \exists \; [v_i,w_i] \in E(T_d), \; 1 \leq i \leq n < \infty,  \text{ such that }S=\bigcup\limits_{1 \leq i \leq n}  T_{[v_i,w_i]} \; \}. \\
 \end{split}
 \end{equation*}
So an element $S \in  \mathcal{HT} $ is by definition the union of a finite number of half sub-trees of $T_d$ and must contain at least one half sub-tree of $T_d$. For $S \in  \mathcal{HT} $ we set $\partial S := \{\eta \in \partial T_d \; \vert \; \eta \text{ an end of } S\}.$
 
Notice, the intersection of two different half sub-trees can either be an edge of $T_d$, or a finite, disjoint union of half sub-trees of $T_d$, or one of the half sub-trees is contained in the other one. 

Recall, we have fixed an end $\xi \in \partial T_d$ with the property that  $\G_{\xi}= \G_\xi^{0}$. This implies for every $g \in \G$  the end $g.\xi \in \G.\xi$ also has the property $\G_{g.\xi}= \G_{g.\xi}^{0}$. Notice, $\G_{g.\xi}=g\G_{\xi}g^{-1}$. 

By the definition of $\G$ it is easy to see $\G. \xi$ is dense in $\partial T_d$. Moreover, we denote $$\partial_{\G.\xi}S: = \G.\xi \cap \partial S, \text{ for every } S \in \mathcal{HT}.$$

\begin{definition}
\label{def::sign_measures}
An \textbf{additive $\CC$-pseudo-measure with respect to $\xi$} on $\mathcal{HT}$ is a function 
\begin{equation*}
\begin{split}
& \mu : \mathcal{HT} \to \CC, \; \; S \in \mathcal{HT} \mapsto \mu(\partial_{\G.\xi} S) \in \CC \\
& \mu(\bigcup\limits_{1 \leq i \leq n} \partial_{\G.\xi} S_i) = \sum\limits_{i=1}^{n} \mu(\partial_{\G.\xi} S_i) \\
& \; \;  \text{  for every pairwise disjoint } \{\partial_{\G.\xi} S_i\}_{1 \leq i\leq n}, \; n < \infty, \; \{S_i\}_{1 \leq i \leq n} \subset \mathcal{HT}.\\
\end{split}
\end{equation*}
We set $\mathcal{M}_{\CC, \G.\xi}(\mathcal{HT}):= \{\mu \text{ additive $\CC$-pseudo-measure  with respect to $\xi$ on } \mathcal{HT}\}.$
\end{definition}

Examples of additive $\CC$-pseudo-measures with respect to $\xi$ on $\mathcal{HT}$ are coming from functions on $\G. \xi$. More precisely, let $f:\G. \xi \to \CC$ be with finite support in $\G. \xi$.  Then the map $\mu_f : \mathcal{HT} \to \CC$ given by $ S \in \mathcal{HT} \mapsto \mu_{f}(\partial_{\G.\xi} S) :=\sum\limits_{\eta \in \G.\xi \cap \partial S} f(\eta) $ is indeed an additive $\CC$-pseudo-measure with respect to $\xi$ on $\mathcal{HT}$. 

The set $\mathcal{M}_{\CC, \G.\xi}(\mathcal{HT})$ is a $\CC$-vector space: 
\begin{enumerate}
\item
for every $\mu_1, \mu_2 \in \mathcal{M}_{\CC, \G.\xi}(\mathcal{HT})$ we define $\mu_1 + \mu_2$ by $S \in \mathcal{HT} \mapsto (\mu_1 + \mu_2)(\partial_{\G.\xi} S) : = \mu_1(\partial_{\G.\xi} S) + \mu_2(\partial_{\G.\xi} S)$, and this gives an additive $\CC$-pseudo-measure with respect to $\xi$ on $\mathcal{HT}$
\item
for every $a \in \CC$ and every $\mu \in \mathcal{M}_{\CC, \G.\xi}(\mathcal{HT})$, we define $a\cdot \mu$ by $S \in \mathcal{HT} \mapsto (a \cdot \mu)(\partial_{\G.\xi} S) : = a \cdot \mu(\partial_{\G.\xi} S)$, that is an additive $\CC$-pseudo-measure with respect to $\xi$ on $\mathcal{HT}$.
\end{enumerate}

As $\G$ acts on $T_d$, there are canonical actions of $\G$ on $\mathcal{HT}$ and on $\mathcal{M}_{\CC, \G.\xi}(\mathcal{HT})$.  For any $g \in \G$ and $S \in \mathcal{HT}$, $g.S \subset T_d$ is again an element of $\mathcal{HT}$. Then for every $\mu \in \mathcal{M}_{\CC, \G.\xi}(\mathcal{HT})$  and every $g \in \G$ we define 
\begin{equation}
\label{equ::G_acts_HT}
g.\mu \text{ by }S \in \mathcal{HT} \mapsto (g.\mu)(\partial_{\G.\xi} S):= \mu(\partial_{\G.\xi} (g^{-1}.S))= \mu(g^{-1}.\partial_{\G.\xi} S),
\end{equation}
and $g.\mu$ is again an element of $\mathcal{M}_{\CC, \G.\xi}(\mathcal{HT})$.

\begin{remark}
\label{rem::homothetic}
As the end $\xi \in \partial T_d$ does not admit hyperbolic elements in $\G$ that fix it, the action of $\G$ on $\G.\xi$ is without ``homothetic transformations''.  By this we mean if $S$ is a half sub-tree of $T_d$ with $\xi \in \partial S$ (thus $\partial S$ is an open standard neighborhood of $\xi$ with respect to cone topology on $T_d \cup \partial T_d$),  then there is no $g \in \G$ with the property that $g. \partial S \subsetneq  \partial S$. This would not be the case if $\xi$ was an end admitting hyperbolic elements in $\G$ fixing it.

Moreover, for every end $\eta \in \G.\xi$, where $\xi$ is our fixed end, any of its horoballs is set-wise stabilized by every $g \in \G_{\eta}= \G_{\eta}^{0}$. Thus, none of the those horoballs is contracted or expended by the action of $\G$. This is not the case if $\xi$ is an end admitting hyperbolic elements in $\G$ fixing it. Therefore, for $\xi$ such that $\G_{\xi}= \G_{\xi}^{0}$, one can say $\G$ acts on the orbit $\G. \xi$ by a sort of ``isometries''. Therefore, the case when $\xi$ is such that $\G_{\xi}= \G_{\xi}^{0}$ is very different than the case when $\xi$ admits  hyperbolic elements in $\G$ fixing it.
\end{remark}

\subsection{A Hilbert space with a family of inner products}

On the $\CC$-vector space $\mathcal{M}_{\CC, \G.\xi}(\mathcal{HT})$ we define a family of inner products that is indexed by the elements of  $\mathcal{HT}$: 
\begin{enumerate}
\item[]
for a given $S \in \mathcal{HT}$ and any $\mu_1,\mu_2 \in \mathcal{M}_{\CC, \G.\xi}(\mathcal{HT})$ we set $$\langle \mu_1, \mu_2 \rangle_{S}:= \mu_1(\partial_{\G.\xi} S) \cdot  \overline{\mu_2(\partial_{\G.\xi} S)}$$ and call $\langle \cdot, \cdot \rangle_{S}$ as \textbf{$S$-inner product} on $\mathcal{M}_{\CC, \G.\xi}(\mathcal{HT})$.
\end{enumerate}

\begin{remark}
\label{rem::def_S_scalar_product}
We list the following easy remarks about $S$-inner products on $\mathcal{M}_{\CC, \G.\xi}(\mathcal{HT})$:
\begin{enumerate}
\item
For every $S \in \mathcal{HT}$, the $S$-inner product verifies the linearity and conjugate symmetry. Indeed, for all $\mu_1,\mu_2, \mu \in \mathcal{M}_{\CC, \G.\xi}(\mathcal{HT})$ and every $a \in \CC$ we have:
\begin{equation*}
\langle a \cdot \mu_1,\mu \rangle_{S}=  (a\cdot \mu_1)(\partial_{\G.\xi} S) \cdot \overline{\mu(\partial_{\G.\xi} S)}= a \cdot \mu_1(\partial_{\G.\xi} S) \cdot \overline{\mu(\partial_{\G.\xi} S)} = a \cdot  \langle \mu_1,\mu \rangle_{S}.
\end{equation*}

\begin{equation*}
\langle \mu_1 + \mu_2,\mu \rangle_{S}= (\mu_1(\partial_{\G.\xi} S) + \mu_2(\partial_{\G.\xi} S)) \cdot \overline{\mu(\partial_{\G.\xi} S)} = \langle \mu_1 ,\mu \rangle_{S} + \langle \mu_2,\mu \rangle_{S}.
\end{equation*}

\begin{equation*}
\langle \mu_1 , \mu_2 \rangle_{S} =  \mu_1(\partial_{\G.\xi} S) \cdot \overline{\mu_2(\partial_{\G.\xi} S)} = \overline{\overline{\mu_1(\partial_{\G.\xi} S) }\cdot \mu_2(\partial_{\G.\xi} S)} = \overline{\langle \mu_2 , \mu_1 \rangle_{S}}.
\end{equation*}

\item
Let $S \in \mathcal{HT}$. If $\mu_1=\mu_2 \in \mathcal{M}_{\CC, \G.\xi}(\mathcal{HT})$ then one can easily notice
$$0 \leq \langle \mu_1, \mu_1\rangle_{S}= \mu_1(\partial_{\G.\xi} S) \cdot \overline{\mu_1(\partial_{\G.\xi} S)}= \vert  \mu_1(\partial_{\G.\xi} S) \vert^{2}  < \infty.$$ 
Then for any $\mu \in \mathcal{M}_{\CC, \G.\xi}(\mathcal{HT})$ and any $S \in \mathcal{HT}$ we can denote $\vert \mu \vert_{S} :=  \sqrt{\langle \mu , \mu \rangle_S}$.  One can easily verify $\vert \mu_1 + \mu_2 \vert_{S} \leq \vert \mu_1\vert_{S} + \vert \mu_2 \vert_{S} $, for every $\mu_1,\mu_2 \in \mathcal{M}_{\CC, \G.\xi}(\mathcal{HT})$.  Therefore, $\vert \cdot \vert_S$ is a seminorm on $\mathcal{M}_{\CC, \G.\xi}(\mathcal{HT})$, for every $S \in \mathcal{HT}$.

\item
Let $S \in \mathcal{HT}$. By the mere definition of the $S$-inner product, for every $\mu_1, \mu_2  \in \mathcal{M}_{\CC, \G.\xi}(\mathcal{HT})$  the Cauchy--Schwarz inequality holds true
$$0 \leq \vert \langle \mu_1, \mu_2 \rangle_{S} \vert \leq \vert \mu_1 \vert_S \cdot \vert  \mu_2\vert_S. $$
\item
From the above one can easily deduce for every $S \in \mathcal{HT}$, the $S$-inner product $\langle \cdot,\cdot \rangle_{S}$ is a positive semi-definite Hermitian form on $\mathcal{M}_{\CC, \G.\xi}(\mathcal{HT})$ with $\vert \cdot \vert_S$ being the associated seminorm.

\end{enumerate}
\end{remark}

\begin{definition}
\label{def::Cauchy_seq}
A sequence $\{\mu_n\}_{n \geq 1} \subset (\mathcal{M}_{\CC, \G.\xi}(\mathcal{HT}), \{\langle \cdot , \cdot \rangle_S\}_{S \in \mathcal{HT}})$ is called \textbf{Cauchy with respect to the family of seminorms $\{\vert \cdot \vert_S\}_{S \in \mathcal{HT}}$} if for every $\epsilon >0$ and every $S \in \mathcal{HT}$, there exists $N = N(\epsilon, S) \geq 1$ such that $ \vert \mu_{n_1} - \mu_{n_2} \vert_S < \epsilon$, for all $n_1,n_2 > N$. 

We say $\{\mu_n\}_{n \geq 1} \subset (\mathcal{M}_{\CC, \G.\xi}(\mathcal{HT}), \{\langle \cdot , \cdot \rangle_S\}_{S \in \mathcal{HT}})$ \textbf{converges to} $\mu \in (\mathcal{M}_{\CC, \G.\xi}(\mathcal{HT}), \{\langle \cdot , \cdot \rangle_S\}_{S \in \mathcal{HT}})$ if for every $\epsilon >0$ and every $S \in \mathcal{HT}$, there exists $N = N(\epsilon, S) \geq 1$ such that $ \vert \mu - \mu_{n} \vert_S < \epsilon$, for every $n > N$. 
\end{definition}

\begin{lemma}
\label{lem::completness}
The $\CC$-vector space $(\mathcal{M}_{\CC, \G.\xi}(\mathcal{HT}), \{\langle \cdot , \cdot \rangle_S\}_{S \in \mathcal{HT}})$ is complete with respect to the family of seminorms $\{\vert \cdot \vert_S\}_{S \in \mathcal{HT}}$.
\end{lemma}
\begin{proof}
Let $\{\mu_n\}_{n \geq 1} \subset (\mathcal{M}_{\CC, \G.\xi}(\mathcal{HT}), \{\langle \cdot , \cdot \rangle_S\}_{S \in \mathcal{HT}})$ be a Cauchy sequence with respect to the family of seminorms $\{\vert \cdot \vert_S\}_{S \in \mathcal{HT}}$. We need to show $\{\mu_n\}_{n \geq 1}$ admits a limit $\mu$ in $(\mathcal{M}_{\CC, \G.\xi}(\mathcal{HT}), \{\langle \cdot , \cdot \rangle_S\}_{S \in \mathcal{HT}})$. Indeed, consider $S \in \mathcal{HT}$.  Then $\{\mu_n\}_{n \geq 1}$ is a Cauchy sequence with respect to the seminorm $\vert \cdot \vert_{S}$. In particular, by the mere definition, we have $\{\mu_n(\partial_{\G.\xi} S)\}_{n \geq 1}$ is a Cauchy sequence in $\CC$ and so it admits a limit in $\CC$ that we denote by $\mu(\partial_{\G.\xi} S)$. We claim $\mu$ is an element of $(\mathcal{M}_{\CC, \G.\xi}(\mathcal{HT}), \{\langle \cdot , \cdot \rangle_S\}_{S \in \mathcal{HT}})$.  By its definition, it is enough to prove $\mu$ is an additive $\CC$-pseudo-measure with respect to $\xi$. But this follows from the properties of limits of sums and the additivity of $\mu_n$.

As $S \in \mathcal{HT}$ is arbitrary, we can conclude $\mu \in (\mathcal{M}_{\CC, \G.\xi}(\mathcal{HT}), \{\langle \cdot , \cdot \rangle_S\}_{S \in \mathcal{HT}})$ and $\{\mu_n\}_{n \geq 1}$ converges to $\mu$ in $(\mathcal{M}_{\CC, \G.\xi}(\mathcal{HT}), \{\langle \cdot , \cdot \rangle_S\}_{S \in \mathcal{HT}})$.
\end{proof}

Let us notice the following geometrical phenomenon.

For $S \in \mathcal{HT}$ we denote by $Conv(S) \subseteq T_d$ the (unique) minimal convex set of $T_d$ containing $S$. One can see $\partial S = \partial Conv(S)$, and in particular, $\partial_{\G.\xi} S = \partial_{\G.\xi} Conv(S)$. Then $\G_{Conv(S)}^0:= \{g \in \G \; \vert \; g(Conv(S))=Conv(S) \text{ pointwise}\}$ is a closed subgroup of $\G$ (and thus compact). The subgroups $\G_{Conv(S)}^0$, with $S \in \mathcal{HT}$, are a stronger version of the theory of root group datum associated with groups acting on Bruhat--Tits buildings. Moreover, $\G_{Conv(S)}^0$ is a larger group when the half sub-trees of $S$ are approaching the visual boundary of $T_d$, thus $S$ being smaller and smaller with respect to the cone topology on $T_d \cup \partial T_d$.  Therefore, $\G_{Conv(S)}^0$ behaves like a $p$-adic number that can be very small in $\QQ_p$, with respect to the $p$-adic norm, but infinite in $\RR$.

\begin{lemma}
\label{lem::root_inf_index}
Let $S \in \mathcal{HT}$ and $\eta \in \partial T_d - \partial S$. Then $\G_{Conv(S)}^0 \cap \G_{\eta}$ has infinite index in $\G_{Conv(S)}^0$.
\end{lemma}
 
\begin{proof}
For simplicity we will assume $S = Conv(S)$ in what follows.
As $\eta \in \partial T_d - \partial S$, there is a unique geodesic ray $[y_0,\eta) \subset T_d$ such that $[y_0,\eta) \cap S = \{y_0\}$. We denote the consecutive vertices of $[y_0,\eta)$ by $\{y_j\}_{j\geq 0}$. It is clear that for every $1 \leq j$ the group $\G_{S \cup [y_0,y_j]}^0$ has finite index in $\G_{S \cup [y_0,y_{j-1}]}^0$, and $$[\G_{S}^0 : \G_{S \cup [y_0,y_{j-1}]}^0 ] \leq [\G_{S}^0 : \G_{S \cup [y_0,y_{j}]}^0 ] < \infty.$$

Suppose $\G_{S}^0 \cap \G_{\eta}$ has finite index in $\G_{S}^0$. Then $\G_{S}^0 \cap \G_{\eta} = \G_{S \cup [y_0,\eta)}^0$ and by our assumption $[\G_{S}^0 : \G_{S \cup [y_0,\eta)}^0] < \infty$. This implies that for some $ n \geq 0$ we have $\G_{S \cup [y_0,y_{n}]}^0 = \G_{S \cup [y_0,y_{j}]}^0$, for every $j \geq n$. 

To obtain a contradiction we proceed as in the proof of Lemma \ref{lem::Stab_end_notcompact} and claim for every $j \geq n$, we have $F_{i([y_j,y_{j+1}])} \leq F_{i([y_{j+1},y_{j+2}])}$, where $F_{i([y_{j},y_{j+1}])}$ is the stabilizer in $F$ of the color $i([y_{j},y_{j+1}]) \in \{1,..., d\}$. Indeed, by the definition of $\G$ one can construct elements $g \in \G$ such that $g([y_j,y_{j+1}])= [y_j,y_{j+1}]$ and $g \in \G_{S \cup [y_0,y_{j+1}]}^0$. If moreover, $F_{i([y_j,y_{j+1}])}$ was not a subgroup  of  $F_{i([y_{j+1},y_{j+2}])}$, then one would be able to construct $g\in  \G_{S \cup [y_0,y_{j+1}]}^0$ such that $g([y_{j+1},y_{j+2}]) \neq [y_{j+1},y_{j+2}]$ and thus $\G_{S \cup [y_0,y_{j+2}]}  < \G_{S \cup [y_0,y_{j+1}]}$, given a contradiction with $\G_{S \cup [y_0,y_{n}]}^0 = \G_{S \cup [y_0,y_{j}]}^0$, for every $j \geq n$. Therefore, the claim follows.

Since the geodesic ray $[y_0,\eta)$ is $d$-colored, there are at least two different colors $k_1 \neq k_2 \in \{1,...,d\} $ on this geodesic ray such that $F_{k_1}=F_{k_2}$. As in the proof of Lemma \ref{lem::Stab_end_notcompact}, choose an edge $e_1 \subset [y_0,\eta)$ colored $k_1$ and consider the unique bi-infinite geodesic line $\ell$ colored only with $k_1$ and $k_2$, alternatively. Then $\G_{\ell}:=\{ g \in \G \; \vert \; g \text{ stabilizes } \ell \}$ contains $\G_{e_{1}}$, and is a non-compact, proper open subgroup in $\G$, given a contradiction. The lemma is proven.
 \end{proof}
 
Therefore, for $S_1 \neq S_2 \in \mathcal{HT}$ (in particular, $\partial S_1 \neq \partial S_2$), the subgroup $ \G_{Conv(S_1)}^0 \cap \G_{Conv(S_2)}^{0}$ has infinite index in both $\G_{Conv(S_1)}^0$ and $\G_{Conv(S_2)}^0$, and thus it has zero measure with respect to the Haar measure of each of the compact groups $\G_{Conv(S_1)}^0,\G_{Conv(S_2)}^0$. 

\subsection{The regular horocyclic unitary representation of $\G$}

We are now ready to define an action of $\G$ on the complex Hilbert space $(\mathcal{M}_{\CC, \G.\xi}(\mathcal{HT}), \{\langle \cdot , \cdot \rangle_S\}_{S \in \mathcal{HT}})$. This is given as follows:
\begin{equation*}
\begin{split}
 g \in \G \mapsto \; &\pi(g): \mathcal{M}_{\CC, \G.\xi}(\mathcal{HT}) \to \mathcal{M}_{\CC, \G.\xi}(\mathcal{HT}) \\
& \pi(g) \mu := g.\mu, \text{ for any } \mu \in \mathcal{M}_{\CC, \G.\xi}(\mathcal{HT}), \\
\end{split}
\end{equation*}
where $g.\mu$ is defined in (\ref{equ::G_acts_HT}).
Then notice for 	any $S \in \mathcal{HT}$ and any $g \in \G$ one has 
$$\langle \mu_1, \mu_2 \rangle_{S} = \langle \pi(g) \mu_1 , \pi(g) \mu_2 \rangle_{g.S}, \text{ for every } \mu_1, \mu_2 \in \mathcal{M}_{\CC, \G.\xi}(\mathcal{HT})$$
therefore obtaining a $\G$-action on the family of inner products  $\{\langle \cdot , \cdot \rangle_S\}_{S \in \mathcal{HT}}$. With this $\G$-action on $\{\langle \cdot , \cdot \rangle_S\}_{S \in \mathcal{HT}}$, for $g \in \G$ the map $$\pi(g) : (\mathcal{M}_{\CC, \G.\xi}(\mathcal{HT}), \{\langle \cdot , \cdot \rangle_S\}_{S \in \mathcal{HT}}) \to (\mathcal{M}_{\CC, \G.\xi}(\mathcal{HT}), \{\langle \cdot , \cdot \rangle_S\}_{S \in \mathcal{HT}}) \text{ is unitary}. $$

\begin{definition}
\label{def::horo_unitary_rep}
The unitary representation defined above $$\pi : \G \to \mathcal{U}((\mathcal{M}_{\CC, \G.\xi}(\mathcal{HT}), \{\langle \cdot , \cdot \rangle_S\}_{S \in \mathcal{HT}}))$$ is called the \textbf{regular horocyclic unitary representation} of $\G$.
\end{definition}

\begin{lemma}
\label{lem::strong_cont_horo}
Let $F$ be primitive. Then the regular horocyclic unitary representation $(\pi, \mathcal{M}_{\CC, \G.\xi}(\mathcal{HT}), \{\langle \cdot , \cdot \rangle_S\}_{S \in \mathcal{HT}})$ of $\G$ is $\G$-strongly continuous: for every $\mu \in \mathcal{M}_{\CC, \G.\xi}(\mathcal{HT})$ the map $g \in \G \mapsto \pi(g)\mu \in (\mathcal{M}_{\CC, \G.\xi}(\mathcal{HT}), \{\langle \cdot , \cdot \rangle_S\}_{S \in \mathcal{HT}})$ is continuous with respect to the topology of $\G$ and the topology on $\mathcal{M}_{\CC, \G.\xi}(\mathcal{HT})$ given by the seminorms  $\{\vert \cdot \vert_S\}_{S \in \mathcal{HT}})$.
\end{lemma}

\begin{proof}
Let $\mu \in \mathcal{M}_{\CC, \G.\xi}(\mathcal{HT})$. It is enough to prove for every $\epsilon >0$ and every $S \in \mathcal{HT}$ there is an open neighborhood $U= U(\epsilon, S)$ of $e \in \G$, with respect to the compact-open topology on $\G$, such that $\vert \mu - \pi(g)\mu\vert_S < \epsilon$, for every $g \in U$. As $S$ is a union of $1 \leq n <\infty$ half sub-trees $T_{[v_i,w_i]}$ of $T_d$, for $1 \leq i \leq n$, we consider a connected finite subtree $B$ of $T_d$ that contains the edges $[v_i,w_i]$, for every $1 \leq i \leq n$. Then  $\G_{B}:= \{g \in \G \; \vert \;  g(B)=B \text{ point-wise}\}$ is an open compact neighborhood of $e \in \G$. As every element $g \in \G_{B}$ fixes point-wise the edge $[v_i,w_i]$, for every $1 \leq i \leq n$, then $g$ will preserve (but not necessarily point-wise fix) the ideal boundary of each of the half sub-trees  $T_{[v_i,w_i]}$. Therefore, for every $g \in \G_{B}$ we have that $g^{-1}.\partial S = \partial S$ set-wise (but not necessarily point-wise), in particular, $g^{-1}.\partial_{\G.\xi} S = \partial_{\G.\xi} S$. And so, by the definitions we have $(\pi(g)\mu)(\partial_{\G.\xi} S)= \mu(g^{-1} . \partial_{\G.\xi} S)= \mu(\partial_{\G.\xi} S)$ and by the definition of $\vert \cdot \vert_S$ we have $\vert \mu - \pi(g)\mu\vert_S =0  < \epsilon$, for every $g \in \G_{B}$. The lemma is proven.
\end{proof}

\begin{lemma}
\label{lem::nonzero_G_inv_vect}
Let $F$ be primitive. Then the regular horocyclic unitary representation $(\pi, \mathcal{M}_{\CC, \G.\xi}(\mathcal{HT}), \{\langle \cdot , \cdot \rangle_S\}_{S \in \mathcal{HT}})$ of $\G$ does not admit non-zero $\G$-invariant vectors.
\end{lemma}
\begin{proof}
Suppose there is $\mu \in \mathcal{M}_{\CC, \G.\xi}(\mathcal{HT}) \setminus \{0\}$ such that $\pi(g)\mu = g.\mu = \mu$ for every $g \in \G$. This means, for every $S \in \mathcal{HT}$ and every $g \in \G$ we have $$\mu(\partial_{\G.\xi} S) = (g.\mu)(\partial_{\G.\xi} S) = \mu(g^{-1}.\partial_{\G.\xi} S)= \mu(\partial_{\G.\xi} g^{-1}.S).$$

Take $S \in \mathcal{HT}$ to be a half sub-tree $T_{[v,w]}$ and in it consider all the edges at distance $2$ from $v$, thus at distance $1$ from $w$. The number of those edges is $(d-1)^2$ and denote this set by $E(2,v) \subset T_{[v,w]}$. As $\G$ is edge-transitive on $T_d$, for each edge $e \in E(2,v)$ there is a hyperbolic element $\gamma_e \in \G$ such that $\gamma_e([v,w])=e$ and with the edges $[v,w],e$ being both in the translation axis of $\gamma_e$. Then $\gamma_e. T_{[v,w]}$ is a proper sub-tree of $T_{[v,w]}$ and notice for every $e_1 \neq e_2 \in E(2,v)$ we have $\gamma_{e_1}. T_{[v,w]} \cap \gamma_{e_2}. T_{[v,w]} = \emptyset$. One can also notice $\partial T_{[v,w]}$ is the finite disjoint union $\bigcup\limits_{e \in E(2,v)} \partial \gamma_e.T_{[v,w]}$.  By our assumption we then have 
$$\mu(\partial T_{[v,w]}) =\mu(\bigcup\limits_{e \in E(2,v)} \partial \gamma_e.T_{[v,w]}) =\sum\limits_{e \in E(2,v)} \mu(\partial \gamma_e.T_{[v,w]})= (d-1)^2 \cdot \mu(\partial T_{[v,w]}))$$
implying $\mu(\partial T_{[v,w]})=0$ and thus $\mu\equiv 0$. This gives a contradiction and the lemma follows.
\end{proof}

We also give here two definitions that are used in the next section.

\begin{definition}
\label{def::S_to_infty}
Let $\{S_n\}_{n \geq 1} \subset \mathcal{HT}$. We write $S_n \xrightarrow[n \to\infty]{} \infty$ if for every convex finite set $B \subset T_d$ there is $N= N(B)$ such that for every $n \geq N$ we have $S_n \subset T_d - B$. If this is the case, we say the sequence $\{S_n\}_{n \geq 1} \subset \mathcal{HT}$ \textbf{converges to infinity}. 
\end{definition}

\begin{definition}
\label{def::vanish_infty}
Let $\mu_1,\mu_2 \in (\mathcal{M}_{\CC, \G.\xi}(\mathcal{HT}), \{\langle \cdot , \cdot \rangle_S\}_{S \in \mathcal{HT}})$. We say the \textbf{$(\mu_1,\mu_2)$-matrix coefficient of $\pi$ vanishes at infinity} if for every $\epsilon >0$ there exist a compact $K=K(\epsilon) \subset \G$ and a convex finite set $B=B(\epsilon) \subset T_d$ such that for every $S \in \mathcal{HT}$, with $S \subset T_d - B$, and every $g \in  \G - K$ we have $\vert \langle \pi(g)\mu_1,\mu_2 \rangle_S \vert < \epsilon$.  We say the \textbf{$(\mu_1,\mu_2)$-matrix coefficient of $\pi$ does not vanish at infinity} if there are $\epsilon > 0$, a sequence $\{g_n\}_{n \geq 1} \subset \G$, with $g_n \xrightarrow[n \to\infty]{} \infty$,  and a sequence $\{S_n\}_{n \geq 1} \subset \mathcal{HT}$, with $S_n \xrightarrow[n \to\infty]{} \infty$, such that for every $n \geq 1$ we have $\vert \langle \pi(g_n)\mu_1,\mu_2 \rangle_{S_n} \vert \geq \epsilon$.
\end{definition}

\subsection{The proof of Theorem \ref{thm_iff_HM}}

Recall, we have fixed an end $\xi \in \partial T_d$ with the property that  $\G_{\xi}= \G_\xi^{0}$. This implies for every $g \in \G$  the end $g.\xi \in \G.\xi$ also has the property $\G_{g.\xi}= \G_{g.\xi}^{0}$. Notice,  $\G_{g.\xi}=g\G_{\xi}g^{-1}$. 

Consider the additive $\CC$-pseudo-measure with respect to $\xi$
 \begin{equation}
 \mu_\xi : \mathcal{HT} \to \CC \text{ given by } S \in \mathcal{HT} \mapsto
\mu_\xi(\partial_{\G.\xi} S)= \left\{ \begin{array}{rcl} 1 & \mbox{if} & \xi \in \partial S\\
 0 & \mbox{if} & \xi \notin \partial S\\
\end{array}\right. 
\end{equation}
and notice $\pi(g)\mu_\xi= g.\mu_\xi= \mu_\xi$ for every $g \in \G_{\xi}= \G_\xi^{0}$.

Consider the completion of $span(\G.\mu_\xi)$ in complex Hilbert space $(\mathcal{M}_{\CC, \G.\xi}(\mathcal{HT}), \{\langle \cdot , \cdot \rangle_S\}_{S \in \mathcal{HT}})$ and denote it by $(\mathcal{M}_{\CC, \G.\xi}(\mathcal{HT})^{\xi}, \{\langle \cdot , \cdot \rangle_S\}_{S \in \mathcal{HT}})$. Then the restriction $\pi\vert_{\mathcal{M}_{\CC, \G.\xi}(\mathcal{HT})^{\xi}}$ of the regular horocyclic unitary representation of $\G$ to the Hilbert subspace $\mathcal{M}_{\CC, \G.\xi}(\mathcal{HT})^{\xi}$ is called \textbf{regular $\xi$-horocyclic unitary representation} of $\G$ and is denoted by $(\pi_{\xi},\mathcal{M}_{\CC, \G.\xi}(\mathcal{HT})^{\xi}, \{\langle \cdot , \cdot \rangle_S\}_{S \in \mathcal{HT}})$.

\begin{lemma}
\label{lem::unit_rep_not_vanish}
Let $F$ be primitive but not $2$-transitive. Let  $\xi \in \partial T_d$ with the property that  $\G_{\xi}= \G_\xi^{0}$. Then the regular $\xi$-horocyclic unitary representation $(\pi_{\xi},\mathcal{M}_{\CC, \G.\xi}(\mathcal{HT})^{\xi}, \{\langle \cdot , \cdot \rangle_S\}_{S \in \mathcal{HT}})$ of $\G$ does not admit non-zero $\G$-invariant vectors and the $(\mu_\xi,\mu_\xi)$-matrix coefficient does not vanish at infinity. In particular, for every $g \in \G_{\xi}$ there is a half sub-tree $S=S(g) \in \mathcal{HT}$ such that $\vert \langle \pi_{\xi}(g)\mu_\xi, \mu_\xi \rangle_S \vert =1$.
\end{lemma}
\begin{proof}
That the unitary representation $(\pi_{\xi},\mathcal{M}_{\CC, \G.\xi}(\mathcal{HT})^{\xi}, \{\langle \cdot , \cdot \rangle_S\}_{S \in \mathcal{HT}})$ of $\G$ does not admit non-zero $\G$-invariant vectors follows from Lemma \ref{lem::nonzero_G_inv_vect}. 

The remaining claim of the Lemma  goes as follows. Fix a vertex $v_0 \in T_d$. Let $g \in \G_{\xi}$ and by our assumptions $g$ is an elliptic element of $ \G_{\xi} \leq \G$. Let $[v(g),w(g)] \subset [v_0,\xi)$ be the first edge point-wise fixed by $g$, where $w(g) \in (v(g),\xi)$. So $g([v(g),\xi))=[v(g),\xi)$ point-wise. In particular, the half sub-tree $T_{[v(g),w(g)]}$ is stabilized by $g$ (but not necessarily point-wise fixed) and so $g^{-1}.\partial T_{[v(g),w(g)]}= g.\partial T_{[v(g),w(g)]}=\partial T_{[v(g),w(g)]}$ set-wise. This implies for $S = T_{[v(g),w(g)]}$ we have $\langle \pi_{\xi}(g)\mu_\xi, \mu_\xi \rangle_S= (\pi_{\xi}(g)\mu_\xi) (\partial_{\G.\xi} S) \cdot \overline{\mu_\xi(\partial_{\G.\xi} S)} =  \mu_\xi (g^{-1}.\partial_{\G.\xi} S) \cdot \overline{\mu_\xi(\partial_{\G.\xi} S)} = \mu_\xi (\partial_{\G.\xi} S) \cdot \overline{\mu_\xi(\partial_{\G.\xi} S)} = 1$.  Notice, as $g \in \G_{\xi}  \to \infty$ the associated  $S = T_{[v(g),w(g)]} \to \infty$ as well. This proves the lemma.
\end{proof}

\begin{lemma}
\label{lem::unit_rep_vanish}
Let $F$ be primitive but not $2$-transitive. Let  $\xi \in \partial T_d$ with the property that  $\G_{\xi}= \G_\xi^{0}$, and $\mu \in span(\G.\mu_{\xi})$. Let  $\gamma \in \G$ be a hyperbolic element with attracting and repelling ends $\eta_{-},\eta_{+} \in \partial T_d$. Then for every $S \in \mathcal{HT}$ with $\eta_{+}, \eta_{-} \notin \partial S$ we have $\vert \langle \pi_{\xi}(\gamma^{n})\mu, \mu \rangle_S \vert \xrightarrow[ n \to\infty]{} 0$ and $\langle \pi_{\xi}(\gamma^{- n})\mu, \mu \rangle_S \vert \xrightarrow[ n \to\infty]{} 0$
\end{lemma}
\begin{proof}
As $\mu \in span(\G.\mu_{\xi})$, the support of $\mu$ in $\G.\xi$ is a finite set.  

Let $S \in \mathcal{HT}$ with $\eta_{+}, \eta_{-} \notin \partial S$. It is enough to consider only the case when $S \in \mathcal{HT}$ is a half sub-tree with $\partial S \cap \supp(\mu) \neq \emptyset$, in particular, $\partial_{\G.\xi} S \cap \supp(\mu) \neq \emptyset$. As $\gamma \in \G$ is hyperbolic and $\eta_{+}, \eta_{-} \notin \partial S$, there exists a power $N>0$ such that for every $n \geq N$, we have $\gamma^{n}.S \cap S = \emptyset, \; \gamma^{-n}.S \cap S = \emptyset$, and $\supp(\mu)  \cap \partial  \gamma^{n}.S = \emptyset, \supp(\mu)  \cap \partial\gamma^{-n}.S = \emptyset$. The conclusion follows.
\end{proof}

\begin{proof}[Proof of Theorem \ref{thm_iff_HM}]. Let $F$ be primitive. By \cite{BM00b}, if $F$ is moreover $2$-transitive then the group $\G$ has the Howe--Moore property. Conversely, assume $\G$ has the Howe--Moore property. Then if $F$ was not $2$-transitive, then by Lemma \ref{lem::unit_rep_not_vanish} there would exist a unitary representation of $\G$ without non-zero $\G$-invariant vectors and with matrix coefficients that do not vanish at infinity, contradicting the Howe--Moore property of $\G$.  The theorem is proven.
\end{proof}

Let us give some final remarks about our construction of the regular horocyclic and $\xi$-horocyclic unitary representations of $\G$. 

First of all, the construction of the regular horocyclic unitary representations of $\G$ will not work if one considers $\xi \in \partial T_d$ admitting hyperbolic elements in $\G$ fixing it. This is because the action of $\G$ on $\G.\xi$ will contract or expend the horoballs around every $\eta \in \G. \xi$, as explained in Remark \ref{rem::homothetic}. In particular, the corresponding additive $\CC$-pseudo-measure $\mu_{\xi}$ with respect to $\xi$ that apparently is fixed by the subgroup $\G_{\xi}$, and $\G_{\xi}^{0} \subsetneq \G_{\xi}$, does not make sense to be considered. This is explained by the following lemma.

\begin{lemma}
\label{lem::no_G_xi-inv_measures}
Let $\eta \in \partial T_d$ such that $\G_{\eta}^{0} \subsetneq \G_{\eta}$. Suppose there is an additive $\CC$-pseudo-measure $\mu$ with respect to $\eta$ that is $\G_{\eta}$-invariant, i.e., $g.\mu = \mu$ for every $g \in \G_{\eta}$. Then $\mu$ is $\G$-invariant.
\end{lemma}
\begin{proof}
Let $\gamma \in \G_{\eta}$ that is hyperbolic, with attracting endpoint $\eta$. Consider the associated negative contraction subgroup $U_{\gamma}^{-}$. We claim $h.\mu = \mu$ for every $h \in U_{\gamma}^{-}$. Indeed, let $h \in U_{\gamma}^{-}$ and let $S \in \mathcal{HT}$. Because $\gamma^{n} h \gamma^{-n} \xrightarrow[n \to\infty]{} e$, there exists $N >0$ such that $\gamma^{N} h \gamma^{-N}. S= S$ set-wise. In particular,  we have $$\partial \gamma^{N} h \gamma^{-N}. S = \partial S \;  \text{ and } \; \partial_{\G.\eta} \gamma^{N} h \gamma^{-N}. S = \partial_{\G.\eta} S.$$
This implies
$$ (\gamma^{N} h \gamma^{-N}.\mu)(\partial_{\G.\eta} S)=  \mu(\gamma^{N} h^{-1} \gamma^{-N}.\partial_{\G.\eta} S)= \mu(\partial_{\G.\eta} \gamma^{N} h^{-1} \gamma^{-N}.S) = \mu(\partial_{\G.\eta} S).$$
By applying $\gamma^{-N} $ to the above equality and using that $\gamma. \mu = \mu$ we obtain 
$$(h \gamma^{-N}.\mu)(\partial_{\G.\eta} S)=  (\gamma^{-N}.\mu)(\partial_{\G.\eta} S) \Leftrightarrow (h.\mu)(\partial_{\G.\eta} S)=  \mu(\partial_{\G.\eta} S).$$
As $S \in \mathcal {HT} $ and $h \in U_{\gamma}^{-}$ are arbitrary, we have $h.\mu=  \mu$ and the claim follows.

Now by Lemma \ref{lem::hyp_C_0} we know $\G= \left\langle U^{+}_{\gamma},U^{-}_{\gamma}\right\rangle$, where $U^{+}_{\gamma} \subset \G_{\eta}$ and the lemma is proven.

\end{proof}

\end{document}